\documentclass[]{scrartcl}
\usepackage[T1]{fontenc}
\usepackage[utf8]{inputenc}

\usepackage{amsmath,amsfonts,amsthm,amssymb}
\usepackage[colorlinks,citecolor=blue]{hyperref}
\usepackage[square]{natbib}

\usepackage{changes}

\usepackage{preamble}

\usepackage{cleveref}

\newcommand\norm[1]{\lVert#1\rVert}
\newcommand\bignorm[1]{\bigl\lVert#1\bigr\rVert}
\newcommand\abs[1]{\lvert#1\rvert}

\newcommand\dual[2]{\langle #1, #2\rangle}
\newcommand\bigdual[2]{\bigl\langle #1, #2\bigr\rangle}
\newcommand\scalarprod[2]{( #1, #2)}
\newcommand\N{\mathbb{N}}
\newcommand\R{\mathbb{R}}
\renewcommand\d{\mathrm{d}}
\newcommand\ds{\mathrm{d}s}
\newcommand\dt{\mathrm{d}t}
\newcommand\dx{\mathrm{d}x}
\newcommand\dxt{\dx\,\dt}
\newcommand\meas{\operatorname{meas}}
\newcommand\gph{\operatorname{gph}}
\newcommand\supp{\operatorname{supp}}

\newcommand{\weakly}{\rightharpoonup}

\newcommand\absrho[1]{\abs{#1}_\rho}
\newcommand\absrhon[1]{\abs{#1}_{\rho_n}}

\newcommand\DD{\mathcal{D}}
\newcommand\EE{\mathcal{E}}
\newcommand\GG{\mathcal{G}}
\newcommand\LL{\mathcal{L}}
\renewcommand\SS{\mathcal{S}}

\newcommand\doi[1]{\href{http://dx.doi.org/\detokenize{#1}}{doi: \detokenize{#1}}}

\newtheorem{theorem}{Theorem}[section]
\newtheorem{lemma}[theorem]{Lemma}
\newtheorem{assumption}[theorem]{Assumption}
\newtheorem{corollary}[theorem]{Corollary}
\newtheorem{remark}[theorem]{Remark}

\let\autoref\cref

\colorlet{Changes@Color}{red}

\begin{document}
%%fakesection: Title and abstract etc
\title{Optimal control of a rate-independent evolution equation via viscous regularization}

\author{%
	Ulisse Stefanelli\footnote{
University of Vienna,
Faculty of Mathematics,
Oskar-Morgenstern-Platz 1, 1090 Wien, Austria;
ulisse.stefanelli@univie.ac.at%
}
	\and
	Daniel Wachsmuth\footnote{
	Institute of Mathematics, University of W\"urzburg, Emil-Fischer-Str.\ 40, 97074 W\"urzburg, Germany;  daniel.wachsmuth@mathematik.uni-wuerzburg.de}
	\footnotemark[4]
	\and
	Gerd Wachsmuth%
	\footnote{%
		Technische Universität Chemnitz,
		Faculty of Mathematics, Professorship Numerical Methods (Partial Differential Equations),
		09107 Chemnitz, Germany;
		gerd.wachsmuth@mathematik.tu-chemnitz.de%
	}
	\footnote{Partially supported by DFG grants within the Priority Program SPP~1962 (\emph{Non-smooth and Complementarity-based Distributed Parameter Systems: Simulation and Hierarchical Optimization})}
}

\maketitle

\medskip

\noindent {\bf Abstract.} {\small
We study the optimal control of a rate-independent system that is driven by a convex quadratic energy.
Since the associated solution mapping is non-smooth, the analysis of such control problems is challenging.
In order to
derive optimality conditions, we study the regularization of the problem via a smoothing
of the dissipation potential and via the addition of some viscosity. The resulting
regularized optimal control problem is analyzed. By driving the regularization parameter
to zero, we obtain a necessary optimality condition for the original, non-smooth problem.
}

\medskip



\section{Introduction}
\label{sec:intro}

Let a Lipschitz domain $\Omega\subset\R^d$ and  $T>0$ be given and set $I:=(0,T)$.
We study the optimal control of a non-smooth evolution problem given by the non-smooth
dissipation
\begin{equation}
	\label{eq:dissipation}
	\DD : H_0^1(\Omega) \to \R,
	\quad
	\DD(\dot z)
	:=
	\int_\Omega \abs{\dot z} \, \dx
\end{equation}
and the  quadratic  energy
\begin{equation}
	\label{eq:energy}
	\EE : H_0^1(\Omega) \times L^2(\Omega) \to \R,
	\quad
	\EE(z, g)
	:=
	\int_\Omega \frac12 \, \abs{\nabla z}^2  -  z \, g\, \dx,
\end{equation}
which give rise to the differential inclusion
\begin{equation}
	\label{eq:state_equation}
	0 \in  \partial|\dot z| - \Delta z -g \quad \text{ in }
        H^{-1}(\Omega), \text{ a.e.\@ in }
          I,
%\partial_{\dot z} \DD(\dot z) + \partial_z \EE(z,g) \quad \text{ a.e.\@ on } (0,T),
\end{equation}
 to be complemented  by the initial condition $z(0) = 0$.
Here,  $z \in H^1(I;H_0^1(\Omega))$ has the role of the state
variable, whereas $g \in H^1(I;L^2(\Omega))$ is the control.
The optimal control problem under consideration reads:
\begin{equation}
	\label{eq:optimal_control_problem}
	\tag{P}
	\text{Minimize} \;
	J(z,g)
	\text{ subject to \eqref{eq:state_equation} and $g(0)=0$},
\end{equation}
where $J$ denotes  a suitable objective functional, see
\eqref{eq:objective_functional} below.  The  requirement
$g(0)=0$  arises as  compatibility condition
 implying the {\it stability} of  the initial state $z(0) = 0$.

The  aim  of this article is  to derive  necessary optimality
conditions.  This turns out to be a quite demanding task, even in the basic setting of
\eqref{eq:state_equation}, for the dependence of the state on the
control is non-smooth. This reflects the non-smoothness of the
dissipation, which on the other hand is the trademark of
rate-independent evolution. In this connection, we refer the reader to
the recent monograph by \cite{MielkeRoubicek2015}, where a thorough
discussion of the current state of the art on rate-independent
systems is recorded.

Let us sketch the strategy of our method.
Under  rather mild  assumptions, the optimal control problem
\eqref{eq:optimal_control_problem} admits global solutions.
By letting  $(\bar z,\bar g)$ be locally optimal for the original
optimal control problem,  we find   $\delta>0$ such that
$J(\bar z,\bar g)\le J(z,g)$ for all $(z,g)$ with $\|g-\bar
g\|_{H^1(I;L^2(\Omega))}\le\delta$ and satisfying the constraints
in  \eqref{eq:optimal_control_problem}.
In order to prove necessary optimality conditions to be satisfied by $(\bar z,\bar g)$
we consider the regularized problem
\begin{equation}
 \min J(z,g) + \frac 12\|g-\bar g\|_{H^1(I;L^2(\Omega))}^2
\end{equation}
subject to $\|g-\bar g\|_{H^1(I;L^2(\Omega))} \le \delta$, $g(0)=0$,
and the regularized  problem
\begin{equation}
	\label{eq:state_equation_reg}
	0 =  \partial|\dot z|_\rho - \rho \Delta \dot z - \Delta z -g \quad \text{ in }
        H^{-1}(\Omega), \text{ a.e.\@ in }
          (0,T),  %\partial_{\dot z}\DD_\rho(\dot z) + \partial_z \EE(z,g),
	\quad z(0) = 0.
\end{equation}
Here,  $|\cdot|_\rho$  is a smooth approximation of  the modulus
$|\cdot|$.
The regularized state equation \eqref{eq:state_equation_reg} is smooth. Hence, necessary optimality conditions
for \eqref{eq:regularized_optimal_control_problem} can be derived by standard techniques.
The main challenge is then  to pass to the limit as  $\rho\searrow0$ in the optimality system.

As already mentioned above, the structure of the state equation
\eqref{eq:state_equation} is inspired by the theory of
rate-independent systems. These arise ubiquitously in applications,
ranging from mechanics and electromagnetism to economics and life
sciences, see \cite{MielkeRoubicek2015} besides the classical monographs
\cite{Visintin1994,BrokateSprekels1996,Krejci1996}.
In particular, the presence of the elliptic operator \eqref{eq:state_equation} can
be put in relation with the occurrence of
exchange energy term in micromagnetics \citep{DeSimone-James} or with
gradient plasticity theories \citep{Aifantis}.

Our method is based on regularizing the equation by adding some
viscosity. This relates with the classical  {\it
  vanishing-viscosity} approach to rate-independent systems. Pioneered
by \cite{Efendiev06}, evolutions
of this technology in the abstract setting are in a series of papers by
\cite{Mielke09,Mielke12,Mielke14}. See also \cite{Krejci09} for an existence theory
for discontinuous loadings based on Kurzweil integration.

Vanishing
viscosity has been applied in a number of mechanical contexts ranging
from plasticity with softening \citep{DalMaso08}, generalized materials
driven by nonconvex energies \citep{Fiaschi09}, crack propagation
\citep{Cagnetti08,Knees08,Knees10,Lazzaroni11,Lazzaroni13,Negri10,Toader09},
nonassociative plasticity of Cam-clay \citep{DalMaso11},
Armstrong-Frederick \citep{Francfort13}, cap type
\citep{Babadjian12}, and heterogeneous materials \citep{Solombrino14}. An application to adhesive contact is in
\cite{Roubicek13}, and damage problems via vanishing viscosity are studied in \cite{Knees13,Knees15}.
In all of these settings, the vanishing-viscosity approach has served
as a tool to circumvent non-convexity of the energy toward existence
of solutions. Our aim here is clearly
different for the energy ${\mathcal E}$ is convex. In particular, we exploit
vanishing viscosity in order to regularize the control-to-state
mapping and deriving optimality conditions.

Optimal  control of finite-dimensional rate-independent processes
 has been considered in
\cite{Brokate1987:1,Brokate1988,BrokateKrejci2013} and we witness  an increasing interest for the optimal control of sweeping processes,
see \cite{CastaingMonteiroMarquesRaynaudDeFitte2014,ColomboHenrionHoangMordukhovich2012,ColomboHenrionHoangMordukhovich2015,ColomboHenrionHoangMordukhovich2016}.
In the infinite-dimensional setting, the  available results are
scant.
The existence of optimal controls,  also in combination with
approximations,  was first studied by
\cite{Rindler2008,Rindler2009}
and subsequently  applied in the context of shape memory materials
by
\cite{EleuteriLussardi2014,EleuteriLussardiStefanelli2013,Stefanelli12}.
In these works, no optimality conditions were given.

 To our knowledge, optimality conditions in the time-continuous, rate-independent,
infinite-dimensional setting were firstly derived in
\cite{Wachsmuth2011:1,Wachsmuth2011:2,Wachsmuth2011:4}  in
the context of  quasi-static plasticity,
see also \cite{HerzogMeyerWachsmuth2014}.  Let us however mention
other  works  addressing
optimality conditions for control problem for rate-independent
systems  in combination with
time-discretizations,  namely  \cite{KocvaraOutrata2005,HerzogMeyerWachsmuth2009:2,HerzogMeyerWachsmuth2010:2,AdamOutraraRoubicek2015}.

 The plan of the paper is as follows.
 We firstly derive an optimality system for
\eqref{eq:optimal_control_problem} by means of formal calculations in
\autoref{subsec:formal_optimality}. The argument is then made rigorous
along the paper and brings to the proof of our main result, namely
\autoref{thm:main_theorem}.
The existence of a solution of \eqref{eq:optimal_control_problem} is
 at the core  of \autoref{sec:existence_solution},
see \autoref{lem:existence}.
In \autoref{sec:regularized},
we address the regularization of \eqref{eq:optimal_control_problem}
 instead.
We study the regularized state equation,
and derive an optimality system for the regularized control problem by means of the regularized adjoint equations.
 Eventually,  in \autoref{subsec:limit_process} we pass to the limit  in the
regularized control problem
 and rigorously obtain optimality
conditions for \eqref{eq:optimal_control_problem} in
\autoref{thm:main_theorem}.

\subsection*{Notation}
\label{subsec:notation}
Recall that $\Omega \subset \R^d$ is a bounded  Lipschitz  domain, $T>0$, and let $I:=(0,T)$
and
$Q := I \times \Omega$.
We
work with standard function spaces like $H_0^1(\Omega)$ and $L^2(\Omega)$.
The space $H_0^1(\Omega)$ is equipped with the norm and scalar product
\begin{equation*}
	\scalarprod{z}{w}_{H_0^1(\Omega)} := \int_\Omega \nabla z \cdot \nabla w \, \dx
	\qquad\text{and}\qquad
	\norm{z}_{H_0^1(\Omega)}^2 := \scalarprod{z}{z}_{H_0^1(\Omega)},
\end{equation*}
respectively.
Throughout the text,
$\Delta : H_0^1(\Omega) \to H^{-1}(\Omega)$ denotes the distributional Laplacian
and $\Delta^{-1} : H^{-1}(\Omega) \to H_0^1(\Omega)$ denotes its inverse.

Moreover, we use the Bochner spaces
$L^p(I; H)$, $W^{1,p}(I; H)$, and $H^1(I; H)$,
where $H$ is a Hilbert space.
By $W^{-1,p}(I; H')$ we denote the dual space of $W^{1,p'}(I; H)$,
where $1/p + 1/p' = 1$, $p,p' \in (1,\infty)$.
Since our state equation is equipped with homogeneous initial conditions,
we also use
\begin{equation*}
	H^1_\star(I; H)
	:=
	\{ v \in H^1(I; H) : v(0) = 0\}
	.
\end{equation*}

We will consider optimal control problems with an objective functional of the type
\begin{equation}\label{eq:objective_functional}
 J(z,g)
 := j_1(z) + j_2(z(T)) + \frac12\|g\|_{H^1(I;L^2(\Omega))}^2,
\end{equation}
where the functions $j_1: L^2(I;H_0^1(\Omega)) \to \R$ and $j_2:H_0^1(\Omega)\to \R$ are assumed to be continuously Fréchet differentiable and bounded from below.

\section{Formal derivation of an optimality system}
\label{subsec:formal_optimality}
In this section, we \emph{formally} derive an optimality system.
It is clear that the resulting system may not be a necessary optimality condition.
However, this derivation sheds some light on the situation
and we get an idea what relations can be expected as necessary conditions.

We start by (formally) restating the optimal control problem by
\begin{equation*}
	\begin{aligned}
		\text{Minimize}  \quad & J(z,g) \\
		\text{such that} \quad & (\dot z(t,x),\, g(t,x) + \Delta z(t,x)) \in M \; \forall (t,x) \in (0,T) \times \Omega.
	\end{aligned}
\end{equation*}
Here,
\begin{align*}
	M &:= \gph \partial\abs{\cdot}
	=
	\{
		(u,v) \in \R^2 : v \in \partial\abs{u}
	\}
	\\
	&=
	\bigh(){
		(-\infty, 0] \times \{-1\}
	}
	\cup
	\bigh(){
		\{0\} \times [-1,1]
	}
	\cup
	\bigh(){
		[0, \infty) \times \{+1\}
	}
	.
\end{align*}
The Lagrangian for this optimization problem is given by
\begin{equation*}
	\LL(z, g, q, \xi)
	=
	J(z,g)
	-
	\scalarprod{q}{\dot z}_{L^2(Q)}
	+
	\scalarprod{\xi}{g + \Delta z}_{L^2(Q)}
	.
\end{equation*}
As (formal) optimality conditions, we would expect
\begin{subequations}
	\label{eq:formal_adjoint_system}
\begin{align}
	\label{eq:formal_adjoint_system_1}
	0 &= \frac{\partial}{\partial z} \LL(z, g, q, \xi),
	% = \frac{\partial}{\partial z} J(z, g) - \dot q + \Delta \xi
	\\
	\label{eq:formal_adjoint_system_2}
	0 &= \frac{\partial}{\partial g} \LL(z, g, q, \xi),
	% = \frac{\partial}{\partial g} J(z, g) + \xi
	\\
	\label{eq:formal_adjoint_system_3}
	(-q(t,x), \xi(t,x)) &\in N_M(\dot z(t,x),\, g(t,x) + \Delta z(t,x)).
\end{align}
\end{subequations}
Here, $N_M$ is a normal-cone mapping associated with the closed set $M \subset \R^2$.
Since $M$ is not convex, the different normal cones of variational analysis, namely Fréchet, Clarke, Mordukhovich,
do not coincide.
In particular, by using the Fréchet normal cone, which is the smallest among these,
we would expect
the relations
\begin{subequations}
	\label{eq:formal_sign_conditions}
	\begin{align}
		\label{eq:formal_sign_conditions_1}
		\dot z(t,x) > 0,\;      g(t,x) + \Delta z(t,x)  =  1  &\quad\Longrightarrow\quad   q(t,x) = 0, \\
		\label{eq:formal_sign_conditions_2}
		\dot z(t,x) = 0,\;      g(t,x) + \Delta z(t,x)  =  1  &\quad\Longrightarrow\quad   q(t,x) > 0, \; \xi(t,x) > 0, \\
		\label{eq:formal_sign_conditions_3}
		\dot z(t,x) = 0,\; \abs{g(t,x) + \Delta z(t,x)} <  1  &\quad\Longrightarrow\quad   \xi(t,x) = 0, \\
		\label{eq:formal_sign_conditions_4}
		\dot z(t,x) = 0,\;      g(t,x) + \Delta z(t,x)  = -1  &\quad\Longrightarrow\quad   q(t,x) < 0, \; \xi(t,x) < 0, \\
		\label{eq:formal_sign_conditions_5}
		\dot z(t,x) < 0,\;      g(t,x) + \Delta z(t,x)  = -1  &\quad\Longrightarrow\quad   q(t,x) = 0.
	\end{align}
\end{subequations}
The above equations \eqref{eq:formal_adjoint_system_1}--\eqref{eq:formal_adjoint_system_2} for $q, \xi$ could be written as
\begin{subequations}
	\label{eq:formal_adjoint_equation}
	\begin{align}
		\label{eq:formal_adjoint_equation_1}
		q(T) &= \mrep{j_2'(z(T))}{j_1'(z) + \Delta \xi}  \qquad \text{a.e.\ on } \Omega, \\
		\label{eq:formal_adjoint_equation_2}
		-\dot q &= j_1'(z) + \Delta \xi  \qquad \text{a.e.\ on } Q,\\
		\label{eq:formal_adjoint_equation_3}
		-\ddot g + g +\xi &= \mrep{0}{j_1'(z) + \Delta \xi}  \qquad \text{a.e.\ on } Q.
		% \quad
		% + \text{b.c.}
	\end{align}
\end{subequations}
Here, \eqref{eq:formal_adjoint_equation_3} is equipped with the boundary conditions
$g(0) = \dot g(T) = 0$.
Hence, this formal derivation suggests that
for each local solution $(z,g)$ of \eqref{eq:optimal_control_problem}, there exist
functions $q, \xi$ such that
\eqref{eq:formal_sign_conditions} and \eqref{eq:formal_adjoint_equation} are satisfied.

\section{Unregularized optimal control problem}
\label{sec:existence_solution}
In this section, we give some first results concerning
the optimal control problem \eqref{eq:optimal_control_problem}.
We recall some known results
for the state equation
and prove the existence of solutions to \eqref{eq:optimal_control_problem}.

A concept tailored to rate-independent systems
is the notion of \emph{energetic solutions},
see \cite[Section~1.6]{MielkeRoubicek2015}.
Since the energy \eqref{eq:energy} is convex,
our situation is much more comfortable and we can use
the formulation \eqref{eq:state_equation},
which is strong in time.
Indeed,
for every $g \in H^1(I; H^{-1}(\Omega))$ with $\norm{g(0)}_{L^\infty(\Omega)} \le 1$,
there is a unique energetic solution $\SS(g) := z \in H^1_\star(I; H_0^1(\Omega))$
and this is the unique solution to \eqref{eq:state_equation},
see \cite[Section~1.6.4, Theorem~3.5.2]{MielkeRoubicek2015}.

The requirement
$\norm{g(0)}_{L^\infty(\Omega)} \le 1$ is needed
as a compatibility condition.
Indeed, it ensures that
$\Delta z(0)+g(0) = g(0)$ is in the range of $\partial_{\dot z} \DD = \partial \|\cdot\|_{L^1(\Omega)}$.
Hence, we define
\begin{equation*}
	\GG_0
	:=
	\{ g \in C(\bar I ; H^{-1}(\Omega)) : \norm{g(0)}_{L^\infty(\Omega)} \le 1\}.
\end{equation*}

Due to the quadratic nature of the energy,
it is possible to recast
the state equation as an evolution variational inequality
in the sense of \cite{Krejci1996}.
\begin{lemma}
	\label{lem:evi}
	Let $z \in H_\star^1(I; H_0^1(\Omega))$
	and $g \in H^1(I; H^{-1}(\Omega)) \cap \GG_0$
	be given.
	Then, the state equation \eqref{eq:state_equation} in $H^{-1}(\Omega)$ is
	equivalent to
	\begin{equation}
		\label{eq:evi}
		\dot z
		\in
		N_{\tilde K}^{\textup{Hilbert}}( - z - \Delta^{-1} g )
		\quad\text{in $H_0^1(\Omega)$, a.e.\ on }(0,T)
	\end{equation}
	and to
	\begin{equation}
		\label{eq:evi_dual}
		\dot z
		\in
		N_{K}( \Delta z + g)
		\quad\text{in $H_0^1(\Omega)$, a.e.\ on }(0,T)
		.
	\end{equation}
	Here,
	\begin{equation*}
		N_{\tilde K}^{\textup{Hilbert}}( v )
		:=
		\{
			w \in H_0^1(\Omega)
			:
			\scalarprod{w}{\tilde v - v}_{H_0^1(\Omega)} \le 0 \quad \forall \tilde v \in \tilde K
		\}
	\end{equation*}
	is the (Hilbert space) normal cone of the set
	\begin{equation*}
		\tilde K
		:=
		\{
			w \in H_0^1(\Omega) :
			-1 \le \Delta w \le 1 \text{ a.e.\ in } \Omega
		\}
		=
		\Delta^{-1}(K)
	\end{equation*}
	at $v \in \tilde K$
	and
	\begin{equation*}
		N_K( v )
		:=
		\{
			w \in H_0^1(\Omega) :
			\dual{ w }{\tilde v - v}_{H^{-1}(\Omega), H_0^1(\Omega)} \le 0 \quad \forall \tilde v \in K
		\}
	\end{equation*}
	is the normal cone of the set
	\begin{equation*}
		K := \{ v \in H^{-1}(\Omega) : v \in L^2(\Omega) \text{ and } -1 \le v \le 1 \text{ a.e.\ in } \Omega \}
		=
		\Delta(\tilde K)
	\end{equation*}
	at
	$v \in K$.
\end{lemma}
\begin{proof}
	The assertion follows directly from standard results in convex analysis
	by
	using the definition of the dissipation \eqref{eq:dissipation}
	and of the energy \eqref{eq:energy}.
\end{proof}
The mapping $(-\Delta^{-1} g) \mapsto z$ is also known as the \emph{play operator},
see \cite[Section~I.3]{Krejci1996}.
From \cite[Remark~I.3.10, Theorem~I.3.12]{Krejci1996}
we find the following regularity results
for equation \eqref{eq:evi}.
\begin{lemma}
	\label{lem:mapping_properties_control_state}
	The control-to-state map $\SS$
	is continuous from
	$H^1(I; H^{-1}(\Omega)) \cap \GG_0$
	to
	$H^1_\star(I; H_0^1(\Omega))$
	and
	Lipschitz continuous from
	$W^{1,1}(I; H^{-1}(\Omega)) \cap \GG_0$
	to
	$L^\infty(I; H_0^1(\Omega))$.
\end{lemma}

The next lemma provides the energy equality \eqref{eq:orthogonality},
which will be crucial to prove the consistency
of the regularization in $H^1_\star(I; H_0^1(\Omega))$, cf.\ \autoref{thm:convergence_in_H1}.
\begin{lemma}
	\label{lem:boundedness_unregularized}
	Let $g \in H^1(I; H^{-1}(\Omega)) \cap \GG_0$ be given
	and set $z = \SS(g)$.
	Then, we have
	\begin{equation}
		\label{eq:orthogonality}
		\dual{\dot z(t)}{
			\Delta \dot z(t) + \dot g(t)
		}
		_{H_0^1(\Omega), H^{-1}(\Omega)}
		= 0
		\qquad\text{for a.a.\ } t \in I.
	\end{equation}
\end{lemma}
\begin{proof}
	Using \eqref{eq:evi_dual}
	and
	$\Delta z(s) + g(s) \in K$ for all $s \in I$,
	we find
	\begin{equation*}
		\dual{\dot z(t)}{
			\Delta z(t \pm h) + g(t \pm h)
			-
			(
				\Delta z(t) + g(t)
		)
		}
		_{H_0^1(\Omega), H^{-1}(\Omega)}
		\le 0
	\end{equation*}
	for almost all $t \in I$ and all $h > 0$
	such that $t \pm h \in I$.
	Using Lebesgue's differentiation theorem,
	see \cite[Theorem~II.2.9]{DiestelUhl1977}
	for the version with Bochner integrals,
	we can pass to the limit $h \searrow 0$.
	This yields the claim,
	see also \cite[(I.3.22)(ii)]{Krejci1996}.
\end{proof}
We note that the a-priori energy estimate
\begin{equation}
	\label{eq:energy_estimate}
	\norm{\dot z(t)}_{H_0^1(\Omega)}
	=
	\norm{\Delta \dot z(t)}_{H^{-1}(\Omega)}
	\le
	\norm{\dot g(t)}_{H^{-1}(\Omega)}
	\qquad\text{for a.a.\ } t \in I
\end{equation}
follows immediately from \eqref{eq:orthogonality}.

In order to prove the existence of solutions of the optimal control problem
\eqref{eq:optimal_control_problem},
we need to show a weak continuity result for $\SS$.
Recall, that $H^1(I; L^2(\Omega))$ is \emph{not} compactly embedded in $H^1(I;H^{-1}(\Omega))$,
hence, the following result is not a simple consequence of \autoref{lem:mapping_properties_control_state}.
Similarly, it does not directly follow from Helly's selection theorem,
which would only give pointwise weak convergence of the state variable.
We note that a similar argument was used in \cite[Theorem~2.3, Section~2.3]{Wachsmuth2011:1}.
\begin{lemma}
	\label{lem:weak_continuity}
	Let $\{g_n\}_{n \in \N} \subset H^1(I; L^2(\Omega)) \cap \GG_0$
	be given such that $g_n \weakly g$ in $H^1(I; L^2(\Omega))$.
	Then, $z_n := \SS(g_n) \weakly \SS(g) =: z$ in $H^1_\star(I; H_0^1(\Omega))$
	and
	$z_n \to z$ in $C(\bar I; H_0^1(\Omega))$.
\end{lemma}
\begin{proof}
	The assumptions imply that $g_n(0)\rightharpoonup  g(0)$ in $H^{-1}(\Omega)$.
	Hence, $g(0)$ belongs to $\GG_0$, which makes $z=\SS(g)$ well-defined.
	Due to \eqref{eq:energy_estimate}, the sequence
	$\{z_n\}_{n \in \N}$ is bounded in $H^1_\star(I; H_0^1(\Omega))$.

	From \eqref{eq:evi_dual} we find
	for arbitrary $t \in \bar I$
	\begin{align*}
		\int_0^t
		\dual{\dot z_n}{(\Delta z + g) - (\Delta z_n + g_n)}_{H_0^1(\Omega), H^{-1}(\Omega)}
		\,\ds
		&\le 0
		,
		\\
		\int_0^t
		\dual{\dot z}{(\Delta z_n + g_n) - (\Delta z + g)}_{H_0^1(\Omega), H^{-1}(\Omega)}
		\,\ds
		&\le 0
		.
	\end{align*}
	Adding these inequalities yields
	\begin{equation*}
		\int_0^t
		\scalarprod{\nabla (\dot z_n - \dot z)}{\nabla (z_n - z)}_{L^2(\Omega)}
		\, \ds
		\le
		\int_0^t
		\dual{\dot z_n - \dot z}{g_n - g}_{H_0^1(\Omega), H^{-1}(\Omega)}
		\,\ds
		,
	\end{equation*}
	which gives
	\begin{align*}
		\frac12 \, \norm{z_n(t) - z(t)}_{H_0^1(\Omega)}^2
		&\le
		\int_0^t
		\dual{\dot z_n - \dot z}{g_n - g}_{H_0^1(\Omega), H^{-1}(\Omega)}
		\,\ds
		\\&\le
		\norm{\dot z_n - \dot z}_{L^2(I; H_0^1(\Omega))}
		\,
		\norm{g_n - g}_{L^2(I; H^{-1}(\Omega))}
		.
	\end{align*}
	Owing to \eqref{eq:energy_estimate}, we have
	\begin{equation*}
		\frac12 \, \norm{z_n(t) - z(t)}_{H_0^1(\Omega)}^2
		\le
		\bigh(){
			\norm{g_n}_{H^1(I; H^{-1}(\Omega))}
			+
			\norm{g}_{H^1(I; H^{-1}(\Omega))}
		}
		\,
		\norm{g_n - g}_{L^2(I; H^{-1}(\Omega))}
		.
	\end{equation*}
	Due to the compact embedding $H^1(I; L^2(\Omega)) \hookrightarrow L^2(I; H^{-1}(\Omega))$,
	we can pass to the limit to obtain the convergence $z_n \to z$ in $C(\bar I; H_0^1(\Omega))$.
	Since $\{z_n\}_{n \in \N}$ is bounded in $H^1_\star(I; H_0^1(\Omega))$,
	the weak convergence $z_n\weakly z$ in $H^1_\star(I; H_0^1(\Omega))$ follows.
\end{proof}

Now, we are in the position to prove the existence of solutions of \eqref{eq:optimal_control_problem}.
\begin{lemma}
	\label{lem:existence}
	There exists a (global) optimal control of \eqref{eq:optimal_control_problem}.
\end{lemma}
The proof is standard, but included for the reader's convenience.
\begin{proof}
	We denote by $j$ the infimal value
	of the optimal control problem
	and by $\{(z_n, g_n)\}_{n \in \N}$
	a minimizing sequence.
	By the boundedness
	of $\{g_n\}_{n \in \N}$ in $H^1(I; L^2(\Omega))$
	we obtain the weak convergence of a subsequence (without relabeling) in
	$H^1(I; L^2(\Omega))$
	towards $\bar g$.

	Now, we have $z_n = \SS(g_n) \to \SS(\bar g)$ in $C(\bar I; H_0^1(\Omega))$
	and $z_n(T) \to z(T)$ in $H_0^1(\Omega)$
	due to \autoref{lem:weak_continuity}.
	This implies
	\begin{equation*}
		J(\bar z, \bar g)
		\le
		\liminf_{n \to \infty} J(z_n, g_n)
		=
		j.
	\end{equation*}
	Hence, $(\bar z, \bar g)$ is globally optimal for \eqref{eq:optimal_control_problem}.
\end{proof}

\section{Regularized optimal control problem}
%%fakesubsection: Intro
\label{sec:regularized}

In this section,
we study the regularized optimal control problem.

\subsection{Regularized dissipation}
\label{subsec:regularized_dissipation}
For given parameter $\rho > 0$, let us define the
regularized dissipation by
\begin{equation}
	\label{eq:regularized_dissipation}
	\DD_\rho : H_0^1(\Omega) \to \R,
	\quad
	\DD_\rho(\dot z)
	:=
	\int_\Omega \absrho{\dot z} + \frac\rho2 \, \abs{\nabla \dot z}^2 \, \dx.
\end{equation}
Note that the additional quadratic term in $\DD_\rho$ will add some
viscosity to our state equation.
In the regularization \eqref{eq:regularized_dissipation},
$\absrho{\cdot}$
is a regularized version of the modulus function $\abs{\cdot} : \R \to \R$
satisfying the following assumption:
\begin{assumption}
	\label{asm:regularized_absolute_value}
	The family $\{\absrho{\cdot}\}_{\rho > 0}$ satisfies
	\begin{enumerate}
		\item\label{absrho_1}
			$\absrho{ \cdot }$ is $C^{2,1}(\R,\R)$ and convex,
		\item\label{absrho_2}
			$\absrho{v} = \absrho{-v}$ for all $v \in \R$,
		\item\label{absrho_3}
			$\absrho{ v } = \abs{v}$
			for all $v \in \R$ with $\abs{v} \ge \rho$, and
		\item\label{absrho_4}
			$\absrho{v}'' \le \frac2\rho$ for all $v \in \R$.
	\end{enumerate}
\end{assumption}
Note that this assumption implies
\begin{equation*}
	\absrho{v}' \in [-1,1],
	\quad\text{and}\quad
	\absrho{v}'' \ge 0
	\quad \forall v\in \R
\end{equation*}
by convexity of $\absrho{\cdot}$.

\begin{lemma}\label{lem:absrho}
Let $\absrho{\cdot}$ satisfy \autoref{asm:regularized_absolute_value}.
 Then it holds
\begin{equation}\label{eqabsrhol1}
 |v| \le \absrho{v}\le |v|+\rho \quad \text{ and } \quad \absrho{v}' \, v \ge |v|-\rho  \quad \forall v\in \R.
\end{equation}
\end{lemma}
\begin{proof}
 \crefname{enumi}{Property}{Properties}
 The first inequality follows from convexity and \autoref{absrho_3}. The second inequality obviously holds for
 $|v|\ge \rho$ due to \autoref{absrho_3}. Now let $v \in [-\rho,\rho]$ be given.
 Using the monotonicity of $\absrho{\cdot}'$ due to \autoref{absrho_1},
 we have
 \begin{equation*}
	 \absrho{v}' \, v
	 =
	 \absrho{v}' \, (v - 0)
	 \ge
	 \absrho{0}' \, (v - 0)
	 =
	 0
	 \ge
	 \abs{v} - \rho,
 \end{equation*}
 since $\absrho{0}' = 0$ follows from \autoref{absrho_2}.
\end{proof}

Let us remark that
\autoref{asm:regularized_absolute_value} is satisfied, e.g., by
\begin{equation*}
	\absrho{v}'' := 2\, \rho^{-2} \, \max\bigh(){\rho - \abs{v}, 0},
	\quad
	\absrho{v}' := \int_0^v \absrho{s}'' \, \ds,
	\quad
	\absrho{v} := \rho + \int_{-\rho}^v \absrho{s}' \, \ds.
\end{equation*}

\subsection{Regularized state equation}
\label{subsec:regularized_state}
Let us now discuss the regularized state equation.
In particular,
we will prove the differentiability of the solution map $\SS_\rho$
and show a-priori stability results.

We recall the regularized problem \eqref{eq:state_equation_reg}
\begin{equation*}
	0 = \partial_{\dot z} \DD_\rho(\dot z) + \partial_z \EE(z,g)
	\quad
	\text{in } H^{-1}(\Omega), \text{ a.e.\ on } I, \quad z(0)=0.
\end{equation*}
By using the differentiability of $\absrho{\cdot}$, we obtain
the equivalent formulation
\begin{equation}
	\label{eq:regularized_state_eq}
	\absrho{\dot z}' - \rho \, \Delta \dot z - \Delta z = g
	\quad
	\text{in } H^{-1}(\Omega), \text{ a.e.\ on } I.
\end{equation}
This equation can be written as the system
\begin{subequations}
	\label{eq:regularized_system}
	\begin{align}
		\label{eq:regularized_system_1}
		\dot z &=
		\mrep{w}{\Delta z + g}
		\quad
		\text{in } \mrep{H_0^1(\Omega),}{H^{-1}(\Omega),}
		\quad\text{a.e.\ on } I
		,
		\\
		\label{eq:regularized_system_2}
		- \rho \, \Delta w + \absrho{w}'
		&=
		\Delta z + g
		\quad
		\text{in } H^{-1}(\Omega),
		\quad\text{a.e.\ on } I
		,
	\end{align}
\end{subequations}
equipped with the initial condition $z(0) = 0$.
In order to discuss the solvability of \eqref{eq:regularized_system},
we first analyze the semilinear equation
\begin{equation}
	\label{eq:semilinear_rho}
	- \rho \, \Delta w + \absrho{w}' = v\quad \text{in } H^{-1}(\Omega).
\end{equation}
Due to the monotonicity of $\absrho{\cdot}'$, this equation has a unique weak solution $w\in H^1_0(\Omega)$ for all $v\in H^{-1}(\Omega)$.
Moreover, the solution depends Lipschitz continuously on the right-hand side.
Let us denote by
$T_\rho := (- \rho \, \Delta + \absrho{\cdot}')^{-1}$
the associated solution  mapping,
which is globally Lipschitz continuous from $H^{-1}(\Omega)$ to $H_0^1(\Omega)$ for fixed, positive $\rho$.

Using this mapping, equation \eqref{eq:regularized_system} can be written as
\begin{equation}
	\label{eq:regularized_state_as_ODE}
	\dot z
	=
	T_\rho
	(g + \Delta z)
	\quad
	\text{in } H_0^1(\Omega), \text{ a.e.\ on } I,
\end{equation}
which is an ODE in $H_0^1(\Omega)$. Due to the global Lipschitz continuity of $T_\rho$, we have the following classical result.

\begin{theorem}
Let $\rho>0$ be given.
 For each $g\in L^2(I; H^{-1}(\Omega))$, there exists a unique solution $z\in H^1_\star(I;H^1_0(\Omega))$
 of the regularized state equation \eqref{eq:state_equation_reg}.
 The mapping $\SS_\rho$, which maps $g$ to $z$, is continuous with respect to these spaces.
\end{theorem}
\begin{proof}
The result follows directly from \cite[Satz 1.3, p. 166]{GajewskiGroegerZacharias}.
\end{proof}

In the next step, we will investigate the differentiability of $\SS_\rho$.
Due to the properties of $\absrho{\cdot}'$, the operator $T_\rho$ is Fréchet differentiable from $H^{-1}(\Omega)$ to $H^1_0(\Omega)$.
Let $v,h\in H^{-1}(\Omega)$ be given with $w = T_\rho(v)$.
By standard arguments it can be proven that $y = T_\rho'(v) \, h$ is given as the unique weak solution of
the equation
\begin{equation}
	- \rho \, \Delta y + \absrho{w}'' \, y = h.
\end{equation}
Moreover due to $\absrho{w}''\ge0$, we can bound the norm of $T_\rho'(v)$ uniformly with respect to $v$ by
\[
 \|y\|_{H^1_0(\Omega) } = \|T_\rho'(v)\,h\|_{H^1_0(\Omega) } \le \rho^{-1/2} \, \|h\|_{H^{-1}(\Omega)}.
\]
Hence, the linearized ODE
\[
 \dot \zeta = T_\rho'( g+\Delta z)(h+\Delta\zeta)
	\quad
	\text{in } H_0^1(\Omega), \text{ a.e.\ on } I
\]
with the initial condition $\zeta(0) = 0$
is uniquely solvable provided $g\in L^2(I;H^{-1}(\Omega))$, $z=\SS_\rho(g)$, and $h\in L^2(I;H^{-1}(\Omega))$,
see again \cite{GajewskiGroegerZacharias}.
Summarizing these arguments leads to
the following differentiability result.
\begin{theorem}
	\label{thm:differentiability}
Let $\rho>0$ be given.
	The regularized control-to-state map $\SS_\rho$
	is Fréchet differentiable
	from
	$L^2(I; H^{-1}(\Omega))$
	to
	$H^1_\star(I;H^1_0(\Omega))$.
	The directional derivative
	$\zeta = \SS_\rho'(g) \, h$
	satisfies the system
	\begin{subequations}
		\label{eq:regularized_linearized_system}
		\begin{align}
			\label{eq:regularized_linearized_system_1}
			\dot \zeta &=
			\mrep{\omega}{\Delta \zeta + g}
			\quad
			\text{in } H_0^1(\Omega)
			, \text{ a.e.\ on } I,
			\\
			\label{eq:regularized_linearized_system_2}
			- \rho \, \Delta \omega + \absrho{\dot z}'' \, \omega
			&=
			\Delta \zeta + h
			\quad
			\text{in } H^{-1}(\Omega), \text{ a.e.\ on } I,
			\\
			\zeta(0) &= 0,
		\end{align}
	\end{subequations}
	where $z$ is given by $z=\SS_\rho(g)$.
\end{theorem}

Now, we show a regularized counterpart to the Lipschitz continuity of $\SS$, cf.\ \autoref{lem:mapping_properties_control_state}.
\begin{lemma}
	\label{lem:uniform_lipschitz_regularized}
	Let $g_1, g_2 \in W^{1,1}(I; H^{-1}(\Omega))$ and $\rho > 0$ be given.
	Then it holds
	\begin{equation*}
		\norm{ \SS_\rho(g_2) - \SS_\rho(g_1) }_{C(\bar I;H_0^1(\Omega))}
		\le
		C \, \norm{ g_2 - g_1 }_{W^{1,1}(I;H^{-1}(\Omega))}
	\end{equation*}
	with $C>0$ solely depending on $T$.
\end{lemma}
\begin{proof}
	By testing the state equations \eqref{eq:regularized_state_eq} for $z_1:=\SS_\rho(g_1)$ and $z_2:=\SS_\rho(g_2)$
	by $\dot z_2 - \dot z_1$, integrating over $(0,t)$, and taking the difference, we get
	\begin{multline*}
		\int_0^t \int_\Omega
		\bigh(){ \absrho{\dot z_2}' - \absrho{\dot z_1}' } \bigh(){\dot z_2 - \dot z_1}
		\, \dx \, \ds
		+
		\rho \, \norm{\dot z_2 - \dot z_1}_{L^2(0,t; H_0^1(\Omega))}^2
		\\
		+
		\int_0^t \scalarprod{\nabla(z_2 - z_1)}{\nabla(\dot z_2 - \dot z_1)} \, \ds
		=
		\int_0^t \dual{g_2 - g_1}{\dot z_2 - \dot z_1}_{H^{-1}(\Omega), H_0^1(\Omega)} \, \ds
		.
	\end{multline*}
	Using the monotonicity of $\absrho{\cdot}'$ and $z_1(0) = z_2(0) = 0$,
	we get for all $t\in \bar I$
	\begin{multline*}
		\frac12 \, \norm{z_2(t) - z_1(t)}_{H_0^1(\Omega)}^2
		\le
		\int_0^t \dual{g_2 - g_1}{\dot z_2 - \dot z_1}_{H^{-1}(\Omega), H_0^1(\Omega)} \, \ds
		\\
		\begin{aligned}
		&=
		-\int_0^t \dual{\dot g_2 - \dot g_1}{z_2 - z_1}_{H^{-1}(\Omega), H_0^1(\Omega)} \, \ds
		+ \dual{g_2(t) - g_1(t)}{z_2(t) - z_1(t)}_{H^{-1}(\Omega), H_0^1(\Omega)}
		\\
		&\le \norm{z_2-z_1}_{L^\infty(I;H^1_0(\Omega))}
		\left( \norm{\dot g_2-\dot g_1}_{L^1(I;H^{-1}(\Omega))} +  \norm{g_2-g_1}_{L^\infty(I;H^{-1}(\Omega))}\right)
		.
		\end{aligned}
	\end{multline*}
	Taking the supremum on the left-hand side, we obtain
	\begin{equation*}
		\norm{z_2 - z_1}_{L^\infty(I; H_0^1(\Omega))}
		\le
		2 \left( \norm{\dot g_2-\dot g_1}_{L^1(I;H^{-1}(\Omega))} +  \norm{g_2-g_1}_{L^\infty(I;H^{-1}(\Omega))}\right),
	\end{equation*}
	which shows the assertion.
\end{proof}

As last result in this section,
we provide some a-priori estimates
and, in particular,
provide the boundedness of $z = \SS_\rho(g)$
in $H^1(I; H_0^1(\Omega))$
independent of $\rho$.

\begin{lemma}
	\label{lem:boundedness_H1H1}
	Let $\rho>0$ and $g\in H^1(I;H^{-1}(\Omega)) \cap \GG_0$ be given, and let $z = \SS_\rho(g)$.
	Then it holds $z \in H^2(I; H_0^1(\Omega))$.
	In addition, there is a constant $C>0$ independent of $\rho$ (and $g$) such that
	\begin{equation*}
		\rho \, \norm{\dot z(T)}_{H_0^1(\Omega)}^2 + \norm{z}_{H^1(I;H_0^1(\Omega))}^2
		\le
		C \, \bigh(){ \rho + \norm{g}_{H^1(I;H^{-1}(\Omega))}^2}.
	\end{equation*}
	and
	\[
	 \|\dot z(0)\|_{H^1_0(\Omega)} \le C.
	\]
\end{lemma}
\begin{proof}
	We start by showing
	$\dot z \in H^1(I; H_0^1(\Omega))$.
	Since $T_\rho$ is globally Lipschitz continuous,
	we have
	\begin{align*}
		\norm{\dot z(t_2) - \dot z(t_1)}_{H_0^1(\Omega)}
		&=
		\norm{T_\rho(g(t_2) + \Delta z(t_2)) - T_\rho(g(t_1) + \Delta z(t_1))}_{H_0^1(\Omega)}
		\\&\le
		L_\rho \, \norm{g(t_2) + \Delta z(t_2) - g(t_1) - \Delta z(t_1)}_{H^{-1}(\Omega)}
		.
	\end{align*}
	with a $\rho$-dependent constant $L_\rho$. Since both $\dot g$ and $\Delta \dot z$ are in $L^2(I;H^{-1}(\Omega))$,
	one can prove with the help of finite differences that it holds $\dot z \in H^1(I; H_0^1(\Omega))$.

	Moreover, we obtain $\dot z(0) = T_\rho( \Delta z(0) + g(0) ) = T_\rho(g(0))$ by continuity. Testing the associated semilinear
	elliptic equation by $\dot z(0)$ and using $z(0)=0$ yields
	\[
		\int_\Omega \absrho{\dot z(0)}\,\dot z(0)\,\dx + \rho \, \|\dot z(0)\|_{H^1_0(\Omega)}^2 =  \int_\Omega g(0)\,\dot z(0)\,\dx.
	\]
	By using the second inequality  in  \eqref{eqabsrhol1} for $\absrho{\cdot}$ as well as the assumption $\|g(0)\|_{L^\infty(\Omega)}\le1$ we obtain
	\[
	 \|\dot z(0)\|_{L^1(\Omega)}  + \rho \, \|\dot z(0)\|_{H^1_0(\Omega)}^2 \le  \|\dot z(0)\|_{L^1(\Omega)} + \rho \, \meas(\Omega),
	\]
	which implies
	\begin{equation}\label{eq:estdotz0}
	  \|\dot z(0)\|_{H^1_0(\Omega)}^2 \le  \meas(\Omega).
	\end{equation}
	Now, let us differentiate \eqref{eq:regularized_state_eq} w.r.t.\ $t$
	to obtain
	\begin{equation*}
		\absrho{\dot z}'' \, \ddot z - \rho \, \Delta \ddot z - \Delta \dot z = \dot g.
	\end{equation*}
	Testing with $\dot z$ and integrating, we find
	\begin{equation*}
		\int_Q
		\absrho{\dot z}'' \, \dot z \, \ddot z + \rho \, \nabla \ddot z \cdot \nabla \dot z + \abs{\nabla \dot z}^2
		\,\dxt
		=
		\int_I
		\dual{\dot g}{\dot z}_{H^{-1}(\Omega), H_0^1(\Omega)}
		\,\dt
		.
	\end{equation*}
	Let us  introduce the function
	\begin{equation*}
		f_\rho (r) = \int_0^r \absrho{s}'' \, s \, \ds.
	\end{equation*}
	This construction implies
	\begin{equation*}
		\frac{\d}{\dt}
		f_\rho(\dot z)
		=
		f_\rho'(\dot z) \, \ddot z
		=
		\absrho{\dot z}'' \, \dot z \, \ddot z
		.
	\end{equation*}
	Consequently, we find
	\begin{multline*}
		\int_\Omega f_\rho(\dot z(T)) \, \dx
		+
		\frac\rho2
		\,
		\norm{\dot z(T)}_{H_0^1(\Omega)}^2
		+
		\norm{\dot z}_{L^2(I;H_0^1(\Omega))}^2
		\\
		\le
		\int_I
		\dual{\dot g}{\dot z}_{H^{-1}(\Omega), H_0^1(\Omega)}
		\,\dt
		+
		\int_\Omega f_\rho(\dot z(0)) \, \dx
		+
		\frac\rho2 \, \meas(\Omega)
		,
	\end{multline*}
	where we used in addition the estimate \eqref{eq:estdotz0} of $\dot z(0)$.
	Due to the assumptions on $\absrho{\cdot}$, the auxiliary function $f_\rho$ is bounded,
	and it holds 	$0 \le f_\rho(s) \le \rho$.
	Hence, we obtain
	\begin{equation}\label{eq:uniformzrho}
		\frac\rho2
		\,
		\norm{\dot z(T)}_{H_0^1(\Omega)}^2
		+
		\norm{\dot z}_{L^2(I;H_0^1(\Omega))}^2
		\\
		\le
		\int_I
		\dual{\dot g}{\dot z}_{H^{-1}(\Omega), H_0^1(\Omega)}
		\,\dt
		+
		\frac32 \, \rho \, \meas(\Omega)
		.
	\end{equation}
	Using Young's inequality, we finally obtain
	\begin{equation*}
		\rho \, \norm{\dot z(T)}_{H_0^1(\Omega)}^2 +
		\norm{\dot z}_{L^2(I;H_0^1(\Omega))}^2
		\le
		\norm{\dot g}_{L^2(I;H^{-1}(\Omega))}^2
			+
			3\,\rho \, \meas(\Omega)
		.
	\end{equation*}
	This shows the claim.
\end{proof}
We emphasize that the compatibility condition $g \in \GG_0$, i.e., $\norm{g}_{L^\infty(\Omega)} \le 1$,
is crucial for the validity of \autoref{lem:boundedness_H1H1}.

\subsection{Convergence of the regularization of the state equation}
\label{subsec:convergence_state_equation}
In this section, we show that the sequence $\{\SS_\rho(g)\}_{\rho > 0}$ converges towards the solution $\SS(g)$
of the unregularized system.
\begin{lemma}
	\label{cor:state_equation_conv}
	Let $\{g_n\}_{n \in \N} \in H^1(I;L^2(\Omega)) \cap \GG_0$ be given,
	such that $g_n \weakly g$ in $H^1(I; L^2(\Omega))$.
	Let $\{\rho_n\}_{n\in\N}$ be a positive sequence with $\rho_n \searrow 0$.
	For $n \in \N$  we set $z_n := \SS_{\rho_n}(g_n)$
	and $z: = \SS(g)$.
	Then,
	$z_n \weakly z$ in $H^1_\star(I; H_0^1(\Omega))$
	and
	$z_n \to z$ in $C(\bar I;H_0^1(\Omega))$.

	Moreover, in case $g_n \equiv g$,
	we have the estimate
	\begin{equation}
		\label{eq:state_equation_conv}
		\norm{z - z_n}_{C(\bar I;H^1_0(\Omega))}
		\le
		C \, (1 + \norm{z}_{H^1(I;H_0^1(\Omega))}) \, \rho_n^{1/2}
		,
	\end{equation}
	with $C > 0$ independent of $g, \rho_n$.
\end{lemma}
\begin{proof}
	By \autoref{lem:boundedness_H1H1} we find that the sequence $\{z_n\}_{n \in \N}$ is bounded in $H^1(I; H_0^1(\Omega))$.
	From the state equation \eqref{eq:state_equation},
	we find
	\begin{equation*}
		\dual{\Delta z + g}{ \dot z_n - \dot z}_{H^{-1}(\Omega), H_0^1(\Omega)}
		+
		\norm{\dot z}_{L^1(\Omega)}
		\le
		\norm{\dot z_n}_{L^1(\Omega)}.
	\end{equation*}
	Similarly, by testing the regularized equation \eqref{eq:regularized_state_eq} by $\dot z_n - \dot z$ we obtain
	\[
		\dual{\Delta z_n + g_n + \rho_n \, \Delta \dot z_n}{ \dot z - \dot z_n}_{H^{-1}(\Omega), H_0^1(\Omega)}
		+
		\absrhon{\dot z_n}
		\le
		\absrhon{\dot z}.
	\]
	Adding both inequalities and integrating on $(0,t)$ for $t\in \bar I$ yields
	\begin{align*}
		&
		\rho_n \, \norm{\dot z - \dot z_n}_{L^2(0,t;H_0^1(\Omega))}^2
		+
		\frac12 \, \norm{z(t) - z_n(t)}_{H_0^1(\Omega)}^2
		\\ &\qquad
		\le
		\int_0^t \int_\Omega
		\abs{\dot z_n} - \absrhon{\dot z_n}
		+
		\absrhon{\dot z} - \abs{\dot z}
		\, \dx\,\ds
		+
		\rho_n \, \norm{\dot z}_{L^2(0,t;H_0^1(\Omega))} \, \norm{\dot z - \dot z_n}_{L^2(0,t;H_0^1(\Omega))}
		\\ &\qquad\qquad
		+
		\int_0^t \dual{g - g_n}{\dot z - \dot z_n}_{H^{-1}(\Omega), H_0^1(\Omega)} \, \ds
		.
	\end{align*}
	With \autoref{lem:absrho}, we can estimate the first integral.
	Applying Young's inequality
	we obtain
	\[
		\frac12 \, \norm{z(t) - z_n(t)}_{H_0^1(\Omega)}^2
	%	\\
		\le
		\rho_n \, \meas(Q)
		+
		\frac{\rho_n}4 \, \norm{z}_{H^1(I;H_0^1(\Omega))}^2
		+
		\int_0^t \dual{g - g_n}{\dot z - \dot z_n}_{H^{-1}(\Omega), H_0^1(\Omega)}
		.
	\]
	which proves the convergence claim
	due to $g_n \to g$ in $L^2(I; H^{-1}(\Omega))$, see the end of the proof of \autoref{lem:weak_continuity},
	as well as estimate \eqref{eq:state_equation_conv}.
\end{proof}

\begin{corollary}
	\label{cor:state_equation_conv_different_g}
	Let $\rho > 0$ and let $g, g_\rho \in H^1(I; L^2(\Omega)) \cap \GG_0$
	be given.
	We set $z_\rho := \SS_\rho(g_\rho)$
	and $z: = \SS(g)$.
	Then it holds
	\begin{equation*}
		\norm{ z-z_\rho}_{C(\bar I;H_0^1(\Omega))}
		\le
		C \, \bigh(){  (1 + \norm{z}_{H^1(I;H_0^1(\Omega))}) \, \rho^{1/2}
		+
		 \norm{g - g_\rho}_{W^{1,1}(I; H^{-1}(\Omega))}
		 }.
	\end{equation*}
	with $C>0$ independent of $\rho,g,g_\rho$.
\end{corollary}
\begin{proof}
	Combine
	\cref{lem:uniform_lipschitz_regularized,cor:state_equation_conv}.
\end{proof}

\begin{theorem}
	\label{thm:convergence_in_H1}
	Let $\{g_n\}_{n \in \N} \subset H^1(I; H^{-1}(\Omega)) \cap \GG_0$ be
	such that $g_n \to g$ in $H^1(I; H^{-1}(\Omega))$.
	Let $\{\rho_n\}_{n\in\N}$ be a positive sequence with $\rho_n \searrow 0$.
	Then,
	$z_n := \SS_{\rho_n}(g_n) \to \SS(g) =: z$ in $H^1_\star(I;H_0^1(\Omega))$.
\end{theorem}
\begin{proof}
	By \autoref{cor:state_equation_conv_different_g}, we obtain the convergence $z_n\to z$ in $C(\bar I; H_0^1(\Omega))$.
	Due to \autoref{lem:boundedness_H1H1}, the sequence $\{z_n\}_{n \in \N}$ is bounded in $H^1_\star(I; H_0^1(\Omega))$.
	Thus, it converges weakly towards $z$ in $H^1_\star(I; H_0^1(\Omega))$.

	As in the proof of \autoref{lem:boundedness_H1H1}, see, e.g., \eqref{eq:uniformzrho},
	we obtain
	\begin{equation*}
		\norm{\dot z_{\rho_n}}_{L^2(I;H_0^1(\Omega))}^2
		\\
		\le
		\int_I
		\dual{\dot g_{\rho_n}}{\dot z_{\rho_n}}_{H^{-1}(\Omega), H_0^1(\Omega)}
		\,\dt
		+
		\frac32\, \rho_n \meas(\Omega)
		.
	\end{equation*}
	As $g_{\rho_n}\to g$ in $H^1(I; H^{-1}(\Omega))$ and $\dot z_{\rho_n} \rightharpoonup \dot z$ in
	$L^2(I;H^{-1}(\Omega))$, we can pass to the limit  to find
\begin{equation*}
	\limsup_{n\to\infty} \norm{\dot z_{\rho_n}}_{L^2(I;H_0^1(\Omega))}^2\leq
	\int_I \dual{\dot g}{\dot z}_{H^{-1}(\Omega),H_0^1(\Omega)}  \dt.
\end{equation*}
	Together with \eqref{eq:orthogonality} this implies
	\[
	\limsup_{n\to\infty} \norm{\dot z_{\rho_n}}_{L^2(I;H_0^1(\Omega))}^2\leq
	\int_I \dual{\dot g}{\dot z}_{H^{-1}(\Omega),H_0^1(\Omega)}  \dt
	= \int_I \|\dot z\|_{H_0^1(\Omega)}^2 \dt = \norm{\dot z}_{L^2(I;H_0^1(\Omega))}^2,
	\]
	which shows $\norm{\dot z_n}_{L^2(I;H_0^1(\Omega))} \to \norm{\dot z}_{L^2(I;H_0^1(\Omega))}$.
	The assertion follows from
	the weak convergence of $z_n$ to $z$ in $H^1_\star(I;H_0^1(\Omega))$.
\end{proof}

\subsection{Regularized optimal control problem}
\label{subsec:regularized_optimal_control}

Let  $(\bar z,\bar g)\in H^1_\star(I;H^1_0(\Omega))\times H^1_\star(I;L^2(\Omega))$ be a fixed local solution of \eqref{eq:optimal_control_problem}.
Then there is $\delta>0$ such that $J(\bar z,\bar g)\le J(z,g)$ for all $(z,g)$ with $\|g-\bar g\|_{H^1(I;L^2(\Omega))}\le\delta$ and satisfying \eqref{eq:state_equation}.
Let us consider the relaxed optimal control problem
with the regularized state equation \eqref{eq:regularized_system} as constraint
\begin{equation}
 \label{eq:regularized_optimal_control_problem}
 \tag{P${}_\rho$}
\begin{aligned}
	\text{Minimize} \qquad&
	J(z,g) + \frac 12\|g-\bar g\|_{H^1(I; L^2(\Omega))}^2 \\
\text{ subject to} \qquad& \text{$\|g-\bar g\|_{H^1(I; L^2(\Omega))} \le \delta$, $g(0)=0$,  and \eqref{eq:regularized_system}}.
\end{aligned}
\end{equation}
Note that $\bar g$ is a feasible control for this problem.
With similar arguments as in the proof of \cref{lem:existence} we
can show the existence of global solutions of \eqref{eq:regularized_optimal_control_problem}.

\begin{lemma}
There exists a (global) optimal control of \eqref{eq:regularized_optimal_control_problem}.
\end{lemma}

Due to special construction of \eqref{eq:regularized_optimal_control_problem}, we can prove convergence of global minimizers
to the local solution $(\bar z, \bar g)$.

\begin{theorem}
\label{thm:convergence_minimize}
Let $\{(z_\rho,g_\rho)\}_{\rho>0}$ denote a family of global solutions of \eqref{eq:regularized_optimal_control_problem}.
 Then it holds $g_\rho\to \bar g$ and $z_\rho \to \bar z$ for $\rho\searrow0$ in $H^1_\star(I;L^2(\Omega))$ and $C(\bar I; H_0^1(\Omega))$, respectively.
\end{theorem}
\begin{proof}
 Due to the constraints of \eqref{eq:regularized_optimal_control_problem}, the controls $\{g_\rho\}_{\rho > 0}$ are uniformly
 bounded in the space $H^1_\star(I;L^2(\Omega))$.
 Let now $\{\rho_k\}_{k \in \N}$ with $\rho_k>0$ and $\rho_k \searrow 0$ such that $g_{\rho_k}$ converges
 weakly in $H^1_\star(I;L^2(\Omega))$ to $\hat g$.
 By \autoref{cor:state_equation_conv} the associated sequence $\{z_{\rho_k}\}_{k \in \N}$ converges weakly in $H^1_\star(I; H^1_0(\Omega))$ to $\hat z$,
 with $\hat z = \SS (\hat g)$, thus $(\hat z,\hat g)$ satisfies the state equation \eqref{eq:state_equation}.
 Moreover, $z_{\rho_k}\to \hat z$ in $C(\bar I; H_0^1(\Omega))$.

 For $\rho>0$ let $\bar z_\rho$ denote the solution of the regularized equation \eqref{eq:regularized_system} to the fixed control $\bar g$.
 Then by \autoref{cor:state_equation_conv}, it holds $\bar z_\rho \to \bar z$ in $C(\bar I;H_0^1(\Omega))$.
 This implies the convergence $J(\bar z_\rho,\bar g) \to J(\bar z,\bar g)$.

 The optimality of $g_{\rho_k}$ yields
 \[
  J(z_{\rho_k},g_{\rho_k})  + \frac12\|g_{\rho_k}-\bar g\|_{H^1(I; L^2(\Omega))}^2 \le J(\bar z_{\rho_k},\bar g).
 \]
 Passing to the limit $k\to\infty$ it follows by lower-semicontinuity that
 \[
  J(\hat z,\hat g)  + \frac12\|\hat g-\bar g\|_{H^1(I; L^2(\Omega))}^2 \le J(\bar z,\bar g).
 \]
 The optimality of $(\bar z,\bar g)$ implies $J(\bar z,\bar g)\le J(\hat z,\hat g)$, which yields $\hat g = \bar g$ and $\hat z = \bar z$.
 Moreover, the strong convergence $g_\rho \to \bar g$ in $H^1_\star(I; L^2(\Omega))$
 follows from
 \begin{equation*}
	 \limsup_{k \to \infty} \frac12 \, \norm{ g_{\rho_k} - \bar g}^2_{H^1(I; L^2(\Omega))}
	 \le
	 \limsup_{k \to \infty} \bigh(){ J(\bar z_{\rho_k}, \bar g) - J(z_{\rho_k}, g_{\rho_k}) }
	 =
	 0.
	 \qedhere
 \end{equation*}
\end{proof}

This result shows that for all $\rho>0$  sufficiently small the constraint $\|g-\bar g\|_{H^1(I; L^2(\Omega))} \le \delta$
of \eqref{eq:regularized_optimal_control_problem} is
immaterial.  %not binding anymore.

\begin{remark}
 In view of the assumptions of \autoref{thm:convergence_in_H1}, we could relax the constraint
 in \eqref{eq:optimal_control_problem} and \eqref{eq:regularized_optimal_control_problem}
 on $g(0)$
 to $\|g(0)\|_{L^\infty(\Omega)}\le1$.
\end{remark}

\subsection{Regularized optimality system}
\label{subsec:regularized_optimality}

Let us now turn to the first-order optimality system of  \eqref{eq:regularized_optimal_control_problem}.
At first, we study the regularity of solutions of the adjoint system to \eqref{eq:regularized_system}.
For given $\rho>0$ and $z_\rho\in H^1(I;H_0^1(\Omega))$ it reads
\begin{subequations}
	\label{eq:regularized_adjoint}
	\begin{align}
		\label{eq:regularized_adjoint_1}
		-\dot q_\rho &=
		\Delta \xi_\rho + j_1'(z_\rho)
		\quad
		&&\text{in } H^{-1}(\Omega), \text{ a.e.\ on } I
		,
		\\
		\label{eq:regularized_adjoint_3}
		q_\rho(T)&=j_2'(z_\rho(T))
		&&\text{in } H^{-1}(\Omega),
		\\
		\label{eq:regularized_adjoint_2}
		- \rho \, \Delta \xi_\rho + \absrho{\dot z_\rho}''\,\xi_\rho
		&= q_\rho
		\quad
		&&\text{in } H^{-1}(\Omega), \text{ a.e.\ on } I
		.
	\end{align}
\end{subequations}
With the help of the adjoint variables we will express derivatives of the reduced objective functional.
Let us first prove existence and uniqueness of solutions.

\begin{lemma}
 Let $(z_\rho,g_\rho)\in H^1(I;H_0^1(\Omega))\times H^1(I; H^{-1}(\Omega))$ be given. Then there exists a unique solution
 $(q_\rho,\xi_\rho) \in H^1(I;H^{-1}(\Omega)) \times H^1(I;H_0^1(\Omega))$
 of the adjoint system \eqref{eq:regularized_adjoint}.
\end{lemma}
\begin{proof}
Using the operator $T_\rho'(g_\rho(t) + \Delta z_\rho(t))$, we can eliminate $\xi_\rho$ with \eqref{eq:regularized_adjoint_2} to rewrite \eqref{eq:regularized_adjoint_1} as a differential equation in $H^{-1}(\Omega)$:
\[
 -\dot q_\rho(t) = \Delta T_\rho'(g_\rho(t) + \Delta z_\rho(t))\,q_\rho(t) + j_1'(z_\rho(t)) \quad \text{in } H^{-1}(\Omega), \text{ f.a.a.\ } t \in I.
\]
By \cite[Satz 1.3, p.\@ 166]{GajewskiGroegerZacharias}, it follows that there exists a uniquely defined solution $q_\rho\in H^1(I;H^{-1}(\Omega))$.
This implies $\xi_\rho := T_\rho'(\cdot)\,q_\rho \in L^2(I;H_0^1(\Omega))$.
Since $v\mapsto \absrho{v}''$ is Lipschitz continuous, we can prove by finite differences the regularity $\xi_\rho \in H^1(I;H_0^1(\Omega))$.
\end{proof}

As a consequence, we can derive first-order optimality conditions for \eqref{eq:regularized_optimal_control_problem}.

\begin{theorem}
\label{thm:regularized_optimality}
 For $\rho>0$, let $(z_\rho,g_\rho)$ be a local optimal solution for
 the regularized optimal control problem \eqref{eq:regularized_optimal_control_problem}
 with $\|g_\rho-\bar g\|_{H^1(I;L^2(\Omega))}< \delta$. Then there exist uniquely determined functions
 $(q_\rho,\xi_\rho) \in H^1(I;H^{-1}(\Omega)) \times H^1(I;H_0^1(\Omega))$
 satisfying \eqref{eq:regularized_adjoint}
 as well as
 \begin{equation}\label{eq:regularize_gradient}
  - 2 \, \ddot g_\rho + \ddot {\bar g} + 2\,g_\rho -\bar g + \xi_\rho =0  \text{ in } H^{-1}(I;L^2(\Omega)), \quad g_\rho(0)=0,\quad
	2 \, \dot g_\rho(T) - \dot{\bar g}(T) = 0.
 \end{equation}
\end{theorem}
\begin{proof}
	Let $h \in H^1_\star(I;L^2(\Omega))$ be arbitrary.
	Since $(z_\rho, g_\rho)$ is locally optimal for \eqref{eq:regularized_optimal_control_problem}
	and $\|g_\rho-\bar g\|_{H^1(I;L^2(\Omega))}< \delta$,
	we have
	\begin{equation*}
		J'(z_\rho, g_\rho) \, (\SS'(g_\rho) \, h, h ) + \scalarprod{g_\rho - \bar g}{h}_{H^1(I; L^2(\Omega))}
		= 0
		.
	\end{equation*}
	Here, we used the Fréchet differentiability of $\SS$ from \autoref{thm:differentiability} and the Fréchet differentiability of $J$.
	We set $\zeta_\rho := \SS'(g_\rho) \, h$.
	Using the structure of $J$, we get
	\begin{equation*}
		j_1'(z_\rho) \, \zeta_\rho + j_2'(z_\rho(T)) \, \zeta_\rho(T) + \scalarprod{2 \, g_\rho - \bar g}{h}_{H^1(I; L^2(\Omega))}
		= 0
		.
	\end{equation*}
	Now, we use the definition of the adjoint variables and \eqref{eq:regularized_adjoint_1}--\eqref{eq:regularized_adjoint_3} implies
	\begin{equation*}
		\int_I \dual{-\dot q_\rho - \Delta\xi_\rho }{\zeta_\rho}_{H^{-1}(\Omega), H_0^1(\Omega)} \, \dt
		+
		\dual{q_\rho(T)}{\zeta_\rho(T)}_{H^{-1}(\Omega), H_0^1(\Omega)}
		+
		\scalarprod{2 \, g_\rho - \bar g}{h}_{H^1(I; L^2(\Omega))}
		= 0
		.
	\end{equation*}
	Via integration by parts, we obtain
	\begin{equation*}
		\int_I
		\dual{q_\rho}{\dot\zeta_\rho}_{H^{-1}(\Omega), H_0^1(\Omega)}
		-
		\dual{\xi_\rho}{\Delta\zeta_\rho}_{H_0^1(\Omega), H^{-1}(\Omega)}
		\, \dt
		+
		\scalarprod{2 \, g_\rho - \bar g}{h}_{H^1(I; L^2(\Omega))}
		= 0
		.
	\end{equation*}
	Using \eqref{eq:regularized_adjoint_2}, we get
	\begin{equation*}
		\int_I
		\dual{- \rho \, \Delta \xi_\rho + \absrho{\dot z_\rho}''\,\xi_\rho}{\dot\zeta_\rho}_{H^{-1}(\Omega), H_0^1(\Omega)}
		-
		\dual{\xi_\rho}{\Delta\zeta_\rho}_{H_0^1(\Omega), H^{-1}(\Omega)}
		\, \dt
		+
		\scalarprod{2 \, g_\rho - \bar g}{h}_{H^1(I; L^2(\Omega))}
		= 0
		.
	\end{equation*}
	Hence, the linearized state equation \eqref{eq:regularized_linearized_system_1}--\eqref{eq:regularized_linearized_system_2}
	yields
	\begin{equation*}
		\scalarprod{\xi_\rho}{h}_{L^2(I; L^2(\Omega))}
		+
		\scalarprod{2 \, g_\rho - \bar g}{h}_{H^1(I; L^2(\Omega))}
		= 0
	\end{equation*}
	for arbitrary $h \in H^1_\star(I; L^2(\Omega))$
	and this is the weak formulation of \eqref{eq:regularize_gradient}.
\end{proof}

Let us now derive bounds on $\xi_\rho$ and $q_\rho$ that are explicit with respect to $\rho$.

\begin{lemma}\label{lem:adjoint_boundedness}
 Let $z_\rho\in C(\bar I;H_0^1(\Omega))$ be given.
Let $(q_\rho,\xi_\rho)$ be the associated solution of \eqref{eq:regularized_adjoint}.
Then there is a constant $C>0$ independent of $\rho$ and $z_\rho$ such that
\begin{multline}
	\label{eq:estimate_q}
\|q_\rho\|_{L^\infty(I; H^{-1}(\Omega))}
+ \rho^{1/2} \, \|\xi_\rho\|_{L^2(I;H_0^1(\Omega))}
 + \bignorm{\absrho{\dot z_\rho}''\,\xi_\rho^2}_{L^1(Q)}^{1/2}
 \\
 \le C\left( \|j_2'(z_\rho(T))\|_{H^{-1}(\Omega)} + \|j_1'(z_\rho)\|_{L^2(I;H^{-1}(\Omega))}\right).
\end{multline}
\end{lemma}
\begin{proof}
Testing \eqref{eq:regularized_adjoint_1} and \eqref{eq:regularized_adjoint_2}  by $(-\Delta)^{-1}q_\rho$ and $\xi_\rho$, respectively, adding both
equations and integrating on $(t,T)$ yields
\begin{multline*}
 \|q_\rho(t)\|_{H^{-1}(\Omega)}^2
 +\int_t^T \int_\Omega \rho \, |\nabla \xi_\rho|^2 + \absrho{\dot z_\rho}''\,\xi_\rho^2 \, \dx \, \ds\\
\begin{aligned}
&= \|j_2'(z_\rho(T))\|_{H^{-1}(\Omega)}^2
  + \int_t^T\dual{j_1'(z_\rho)}{(-\Delta)^{-1}q_\rho}_{H^{-1}(\Omega), H_0^1(\Omega)}\, \ds\\
  &\le\|j_2'(z_\rho(T))\|_{H^{-1}(\Omega)}^2 + \|j_1'(z_\rho)\|_{L^2(I;H^{-1}(\Omega))}^2 +  \int_t^T\|q_\rho(s)\|_{H^{-1}(\Omega)}^2 \, \ds.
  \end{aligned}
\end{multline*}
The claim follows now from Gronwall's inequality.
\end{proof}

It remains to get an estimate for $\xi_\rho$.
\begin{corollary}
	\label{cor:estimate_xi_rho}
	Let $z_\rho\in C(\bar I;H_0^1(\Omega))$ be given.
	Let $(q_\rho,\xi_\rho)$ be the associated solution of \eqref{eq:regularized_adjoint}.
	Then there is a constant $C>0$ independent of $\rho$ and $z_\rho$ such that it holds
	\begin{equation*}
		\norm{ \xi_\rho }_{W^{-1,p}(I; H_0^1(\Omega))}
		\le
		C\left( \|j_2'(z_\rho(T))\|_{H^{-1}(\Omega)} + \|j_1'(z_\rho)\|_{L^p(I;H^{-1}(\Omega))}\right)
		.
	\end{equation*}
	for all $p \in [2,\infty)$.
\end{corollary}
\begin{proof}
	For all $v \in W^{1,p'}(I; H^{-1}(\Omega))$,
	where $1 = 1/p + 1/p'$,
	by \eqref{eq:regularized_adjoint_1} and \eqref{eq:regularized_adjoint_3}
	we have
	\begin{align*}
		&\dual{\xi_\rho}{v}_{W^{-1,p}(I; H_0^1(\Omega)),W^{1,p'}(I; H^{-1}(\Omega))}
		=
		\int_I \dual{\xi_\rho}{v}_{H_0^1(\Omega), H^{-1}(\Omega)} \, \dt
		\\&\qquad
		=
		\int_I
		\bigdual{
			\Delta^{-1} \,
			\bigh(){
				-\dot q_\rho
				+
				j_1'(z_\rho)
			}
		}{v}_{H_0^1(\Omega), H^{-1}(\Omega)}
		\, \dt
		% \\&\qquad
		% =
		% \int_I
		% \bigdual{
		% 	-\Delta^{-1} \,
		% 	\dot q_\rho
		% }{v}_{H_0^1(\Omega), H^{-1}(\Omega)}
		% \, \dt
		% +
		% \int_I
		% \bigdual{
		% 	\Delta^{-1} \,
		% 	j_1'(z_\rho)
		% }{v}_{H_0^1(\Omega), H^{-1}(\Omega)}
		% \, \dt
		\\&\qquad
		=
		\int_I
		\bigdual{
			\Delta^{-1} \,
			q_\rho
		}{\dot v}_{H_0^1(\Omega), H^{-1}(\Omega)}
		\, \dt
		-
		\bigdual{
			\Delta^{-1} \,
			q_\rho(T)
		}{v(T)}_{H_0^1(\Omega), H^{-1}(\Omega)}
		\\&\qquad\qquad
		+
		\bigdual{
			\Delta^{-1} \,
			q_\rho(0)
		}{v(0)}_{H_0^1(\Omega), H^{-1}(\Omega)}
		+
		\int_I
		\bigdual{
			\Delta^{-1} \,
			j_1'(z_\rho)
		}{v}_{H_0^1(\Omega), H^{-1}(\Omega)}
		\, \dt
		\\&\qquad
		\le
		C \,
		\bigh(){
			\norm{q_\rho}_{L^\infty(I; H^{-1}(\Omega))}
			+
			\norm{j_2'(z_\rho(T))}_{H^{-1}(\Omega)}
		}
		\, \norm{v}_{W^{1,1}(I; H^{-1}(\Omega))}
		\\&\qquad\qquad
		+
		\norm{j_1'(z_\rho)}_{L^p(I; H^{-1}(\Omega))} \, \norm{v}_{W^{1,p'}(I; H^{-1}(\Omega))}
		.
	\end{align*}
	Hence, the claim follows from \eqref{eq:estimate_q}.
\end{proof}

\section{Passing to the limit}
\label{subsec:limit_process}

In this final section
we investigate the limit $\rho\searrow0$
and prove our main result,
namely \autoref{thm:main_theorem}.

\begin{lemma}
\label{lem:complementarity}
Let $\{z_\rho\}_{\rho > 0}$ be given such that $z_\rho \to z$ in $H^1_\star(I;H_0^1(\Omega))$.
Let $\{(q_\rho,\xi_\rho)\}_{\rho > 0}$ be the family of solutions of the adjoint system \eqref{eq:regularized_adjoint}.
If $q_\rho\rightharpoonup q$ in $L^2(I;H^{-1}(\Omega))$ then it holds
\[
 \dual{q}{\phi \, |\dot z|} = 0 \quad \forall\phi\in L^2(I;C_0^\infty(\Omega)).
\]
\end{lemma}
\begin{proof}
 Let $\phi\in L^2(I;C_0^\infty(\Omega))$.
 Then it holds $\phi \, |\dot z_\rho|\in L^2(I;H_0^1(\Omega))$, and $\nabla(\phi \, |\dot z_\rho|)$ is bounded in $L^2(Q)$ uniformly with respect to $\rho$.
 Testing \eqref{eq:regularized_adjoint_2} with $\phi \, |\dot z_\rho|$ yields
 \[
  \dual{q_\rho}{\phi \, |\dot z_\rho|} =
	 \int_Q \rho\nabla \xi_\rho \nabla(\phi \, |\dot z_\rho|) +
        \absrho{\dot z_\rho}''\,\xi_\rho\,\phi \, |\dot z_\rho|\,
        \dxt.
 \]
 The first addend on the right-hand side tends to zero for $\rho\searrow0$ as $\rho^{1/2} \, \nabla \xi$ and $\nabla(\phi \, |\dot z_\rho|)$
are bounded in $L^2(Q)$ uniformly in $\rho$, see \autoref{lem:adjoint_boundedness}.
To bound the second addend, observe that $\sqrt{\absrho{\dot z_\rho}''}\,|\dot z_\rho|$ is pointwise bounded by $\rho^{1/2}$
due to \autoref{asm:regularized_absolute_value}.
As $\sqrt{\absrho{\dot z_\rho}''}\,\xi_\rho$ is uniformly bounded in $L^2(Q)$ by \autoref{lem:adjoint_boundedness}, the second addend vanishes for $\rho\searrow0$ as well.
\end{proof}
In particular, \autoref{lem:complementarity} shows that \eqref{eq:formal_sign_conditions_1} and \eqref{eq:formal_sign_conditions_5} hold
in a distributional sense.
We are now in the position to formulate and prove the main result of this article.

\begin{theorem}
	\label{thm:main_theorem}
 Let $(\bar z,\bar g)$ be a local solution of \eqref{eq:optimal_control_problem}. Then there are
 $q\in L^\infty(I;H^{-1}(\Omega))$ and $\xi\in  W^{-1,p}(I;L^2(\Omega))$ (for all $p \in [2,\infty)$)
 such that
\begin{subequations}
	\label{eq:optimality_system}
	\begin{align}
		\label{eq:optimality_system_1}
		-\dot q &=
		\Delta \xi + j_1'(\bar z)
		,
		\quad
		&&%\text{in } ???,
		\\
		\label{eq:optimality_system_2}
		q(T)&=j_2'(\bar z(T)), &&%\text{in ???},
		\\
		\label{eq:optimality_system_3}
		-\ddot {\bar g} + \bar g+ \xi &=0,
		\quad
		\bar g(0) = 0,
		\quad
		\dot{\bar g}(T) = 0
		,
		\\
		\label{eq:optimality_system_4}
		\dual{q}{\phi \, |\dot {\bar z}|} &= 0 &&\forall\phi\in L^2(I;C_0^\infty(\Omega))
	\end{align}
	is satisfied.
Here, \eqref{eq:optimality_system_1}--\eqref{eq:optimality_system_2} have to be understood in the following very weak sense:
For all $\phi \in H^1_\star(\bar I; H_0^1(\Omega))$ it holds
\begin{multline}
		\label{eq:optimality_system_5}
\int_I \dual{q}{\dot \phi}_{H^{-1}(\Omega), H^1_0(\Omega)} \, \dt
-
\dual{\phi}{\Delta \xi}_{H^1(I;H_0^1(\Omega)),H^{-1}(I;H^{-1}(\Omega))}
\\
=
\dual{j_2'(\bar z(T))}{\phi(T)}_{H^{-1}(\Omega), H^1_0(\Omega)}+ \int_I \dual{j_1'(\bar z)}{\phi}_{H^{-1}(\Omega), H^1_0(\Omega)} \, \dt.
\end{multline}
\end{subequations}
\end{theorem}
\begin{proof}
 Let $(z_\rho,g_\rho)$ be a family of global solutions of \eqref{eq:regularized_optimal_control_problem}
 such that $g_\rho\to \bar g$ and $z_\rho \to \bar z$ for $\rho\searrow0$ in $H^1(I;L^2(\Omega))$ and $C(\bar I; H_0^1(\Omega))$, respectively.
 This is possible due to \autoref{thm:convergence_minimize}.
 Let now $(q_\rho,\xi_\rho)$ be the associated adjoint states provided by \autoref{thm:regularized_optimality}.

 Due to \autoref{lem:adjoint_boundedness,cor:estimate_xi_rho}, we find that the dual quantities $(q_\rho,\xi_\rho)$ are
 uniformly bounded in $L^\infty(I;H^{-1}(\Omega))\times W^{-1,p}(I;H_0^1(\Omega))$, respectively.
 Thus we can pass to the limit in \eqref{eq:regularize_gradient} and \eqref{eq:regularized_adjoint_1} to obtain \eqref{eq:optimality_system_3} and \eqref{eq:optimality_system_5}.
\autoref{lem:complementarity} proves \eqref{eq:optimality_system_4}.
\end{proof}

\begin{remark}\hfill
	\label{rem:optimality_conditions}
	\begin{enumerate}
		\item
			In \autoref{thm:main_theorem},
			we have rigorously checked relations \eqref{eq:formal_sign_conditions_1} and \eqref{eq:formal_sign_conditions_5}
			from \eqref{eq:formal_sign_conditions}.
			As a next step, it would be desirable to prove also \eqref{eq:formal_sign_conditions_3}
			in variational terms.
			However, due to the low regularity of the involved quantities,
			$g \in H^1(I; L^2(\Omega))$, $\Delta z \in H^1(I; H^{-1}(\Omega))$
			and
			$\xi \in W^{-1,p}(I; L^2(\Omega))$ it is not clear how \eqref{eq:formal_sign_conditions_3} could be formulated.

			Let us highlight two possible approaches which might be tractable.
			First, one could
			take $\phi \in H^1(I; L^2(\Omega))$ with $\xi = 0$ on the set where $\abs{g + \Delta z} = 1$
			and show $\dual{\phi}{\xi} = 0$.
			However, since $g_\rho + \Delta z_\rho$ does only converge in $H^1(I; H^{-1}(\Omega))$,
			it is not clear how this complementarity could be derived.

			Alternatively,
			one could choose $\varphi \in C^\infty(\R; \R)$ with $\supp(\varphi) \subset [-1,1]$
			and
			show $\dual{\varphi(g + \Delta z)}{\xi} = 0$.
			However, this formulation requires $\varphi(g + \Delta z) \in H^1(I; L^2(\Omega))$,
			thus $z \in H^1(I; H^2(\Omega))$.
			This regularity of $z$, however, seems to be not available.

		\item
			The low regularity of the adjoint variables
                        $q$, $\xi$  is not surprising for it has
                        already been observed in connection with
			%is also known from
other optimality systems for rate-independent evolutions,
			see, e.g.,
			\cite[Satz~8.12]{Brokate1987:1},
			\cite[Theorem~5.2]{BrokateKrejci2013},
			\cite[Theorem~3.1]{Wachsmuth2011:4}.
	\end{enumerate}
\end{remark}

\section{Conclusions and outlook}
\label{sec:concl}

In this work, we have derived
optimality conditions
for the optimal control of a rate-independent process.  The full
set of conditions has been formally derived and we have succeeded in presenting
rigorous arguments for the validity of a specific subset of those.

The verification of  the remaining optimality conditions as well
as their possible validity in a stronger regularity setting
will be the object of further research.
 A  time-discretization
or a decoupling of the smoothing of $\abs{\cdot}$ and the additional
viscosity
 might offer the chance of deriving  the complementarity
\eqref{eq:formal_sign_conditions_3}  as well. Note however that
this task is challenging by the low regularity of the adjoint
variables.

\section*{Acknowledgments}  U.S.\ acknowledges support by the Austrian Science Fund (FWF) projects  P
27052  and I 2375 and the kind hospitality of the Institute of Mathematics
of the University of W\"urzburg, where this research was initiated.
  This work has been funded by the Vienna Science and Technology Fund (WWTF)
through Project MA14-009.

D.W.\ and G.W.\ acknowledge support by DFG grants within the Priority Program SPP~1962
(\emph{Non-smooth and Complementarity-based Distributed Parameter Systems: Simulation and Hierarchical Optimization}).

%%fakesection: Bibliography

\bibliography{references}
\bibliographystyle{plainnat}

\end{document}